\documentclass[12pt]{article}   % If the bibliography is changed:
\newtheorem{theorem}{Theorem}% [section]
\newtheorem{corollary}[theorem]{Corollary}
\newtheorem{conjecture}[theorem]{Conjecture}

\newtheorem{lemma}[theorem]{Lemma}
\newtheorem{example}[theorem]{Example}

\newenvironment {proof} {{\it Proof.}}{\hspace*{\fill}$\Box$\par\vspace{4mm}}
\usepackage{amssymb}
\usepackage{amsmath}
\usepackage{amsfonts}
\usepackage{graphicx,psfrag}
\usepackage{tikz}
\begin{document}

\bibliographystyle{plain}

\title{Some Problems on the Extremal Energy of Integral Weighted Graphs}
\title{On the Extremal Energy of Integral Weighted Graphs}

 \author{Richard A. Brualdi\\
 Department of Mathematics\\
 University of Wisconsin\\
 Madison, WI 53706 USA\\
{\tt brualdi@math.wisc.edu}
 \and
 Jia-Yu Shao\thanks{Research
supported by the National Science Foundation of China (10731040).}
\\
 Department of Mathematics\\
 Tongji University\\
 Shanghai, China\\
 {\tt jyshao@sh163.net}
 \and
 \medskip
Shi-Cai Gong,
%\footnote{Email:scgong@zafu.edu.cn}
 Chang-Qing Xu\thanks{Research supported by the National Science Foundation of China (10871230).},
   Guang-Hui Xu\thanks{Research supported by the Zhejiang Provincial Natural Science Foundation of China (Y7080364)}\\
Department of Mathematics\\
 Zhejiang A \& F University\\
 Lin'An, Hangzhou 311300 China\\
 {\tt \{scgong,cqxurichard,ghxu\}@zafu.edu.cn}\\
 %{\tt ghxu@hotmail.com}\\
 %{\tt g$_$h$_$xu@hotmail.com}\\
 %{\tt cqxurichard@163.com}
 }
 %\and
 %Chang-Qing Xu\\
 %Zhejiang A \& F University\\
 %Linan, China\\
 %{\tt cqxu@zafu.edu.cn}\\
 %\and
 %Shicai Gong\\
 %Zhejiang A \& F University\\
 %Linan, China
 %{\tt scgong@zafu.edu.cn}

\maketitle

 \begin{abstract}
Let  ${\mathcal T}(n,m)$ and  ${\mathcal F}(n,m)$ denote the classes of weighted trees and forests, respectively, of order $n$ with the positive integral weights and the fixed total weight sum $m$, respectively. In this paper, we determine the minimum energies for both the classes  ${\mathcal T}(n,m)$ and  ${\mathcal F}(n,m)$. We also determine the maximum energy for the class  ${\mathcal F}(n,m)$. In all cases, we characterize the weighted graphs whose energies reach these extremal values.  We also solve the similar maximum energy and minimum energy problems for the classes of (0,1) weighted trees and forests.

\vskip 0.3cm

\noindent {\bf Key words:  Energy, graph, weighted graph, integral weights, tree, forest.}
 
 \smallskip
\noindent {\bf AMS subject classifications:  05C50, 05D99, 15A18} 
\end{abstract}

\section{Introduction}

We consider trees and forests on $n$ vertices in which each edge is assigned
a positive integral weight. Such weighted graphs can be regarded as graphs with multiple edges, or {\it multitrees and multiforests}. We assume that the sum of the weights is equal to a specified integer $m$, that is, the total number of edges in the multigraph equals $m$. We let ${\mathcal T}(n,m)$ denote the collection of such multitrees on $n$ vertices with total number of edges equal to $m$. The set of multiforests on $n$ vertices with total number of edges equal to $m$
is similarly denoted by ${\mathcal F}(n,m)$.

In general, the {\it energy} of a multigraph $G$ on $n$ vertices is defined to be
\[\mathbb{E}(G)= \sum_{i=1}^n |\lambda_i|\]
where $\lambda_1,\lambda_2,\ldots,\lambda_n$ are the $n$ (real) eigenvalues of the (nonnegative, integral, symmetric) adjacency matrix $A$ of $G$. Note that, since $A$ has only zeros on its main diagonal, the trace of $A$ equals $0$ and hence ${\rm tr}(A)=\sum_{i=1}^n\lambda_i=0$.  More information on graph eigenvalues can be found in \cite{DMH,CRS}.

 Let
\[\overline{\mathbb{E}}(n,m)=\max\{\mathbb{E}(T): T\in  {\mathcal T}(n,m)\}\]
be the maximum energy of a tree in ${\mathcal T}(n,m)$, and let
\[\widetilde{\mathbb{E}}(n,m)=\min\{\mathbb{E}(T): T\in  {\mathcal T}(n,m)\}\]
be the minimum energy of a tree in ${\mathcal T}(n,m)$.

We  also use the similar notations $\overline{\mathbb{E}}_{F}(n,m)$ and $\widetilde{\mathbb{E}}_{F}(n,m)$ for the class ${\mathcal F}(n,m)$ of multiforests.

\vskip 0.5cm

We consider the following problems concerning the extremal energies of positive integral weighted trees and forests.

\begin{enumerate}
\item[\rm (1)] For a given $m$, determine $\overline{\mathbb{E}}(n,m)$ and the trees in ${\mathcal T}(n,m)$ with this maximum energy. Similar problems can be considered for the class ${\mathcal F}(n,m)$ of weighted forests.
\item[\rm (2)]  For a given $m$, determine $\widetilde{\mathbb{E}}(n,m)$ and the trees in ${\mathcal T}(n,m)$ with this minimum energy. Similar problems can be considered for the class ${\mathcal F}(n,m)$ of weighted forests.
\end{enumerate}

%\vskip 0.5cm

We also consider some subclasses of ${\mathcal T}(n,m)$ or  ${\mathcal F}(n,m)$. For example, we can fix the graph $T$, or fix the weight sequence (in the sense of non-increasing order), or fix both the graph and the weight sequence but the distribution of these weights on the edges can vary. In each of the subclasses, we can consider the corresponding maximum energy problem and the minimum energy problem. 

%\vskip 0.5cm

In this paper, we solve the minimum energy problems for both the classes  ${\mathcal T}(n,m)$ and  ${\mathcal F}(n,m)$. We also solve the maximum energy problem for the class  ${\mathcal F}(n,m)$ of weighted forests. In all these results, we obtain both the extremal values (of the energies) and the characterizations of the weighted graphs whose energies reach these extremal values. We also solve the similar extremal problems for  (0,1) weighted trees and forests.
In addition, for $m\ge n$ we determine the unique tree in ${\mathcal T}(n,m)$ 
with maximum energy for the weight sequence $(m-n+2,1,\ldots,1)$.

Let $G$ be a weighted tree on $n$ vertices. Then its characteristic polynomial can be written as:
\begin{equation}\label{eq:charpoly}
\phi(G,x) = \sum\limits_{k=0}^{\lfloor n/2\rfloor}(-1)^{k}b_{k}(G)x^{n-2k}
\end{equation}
where $b_{k}(G)\ge 0 \ $ for all $k$.
Here $b_k(G)$ is the sum, over all matchings of $G$ of $k$ edges, 
of the  products of the squares of the weights of the edges of the matching.

From the  Coulson integral formula for the energy of graphs \cite{CA,I2,IM,IO}, we conclude that if $G$ is a weighted bipartite graph with the characteristic polynomial as in (\ref{eq:charpoly}), then:
\begin{equation}\label{eq:coulson}
\mathbb{E}(G)=\frac{2}{\pi}\int\limits_0^{+\infty} \frac{1}{x^2}\ln\left(\sum_{k=0}^{\lfloor n/2\rfloor}b_{k}(G)x^{2k}\right) dx.
\end{equation}
Formula (\ref{eq:coulson}) holds for both simple and weighted bipartite graphs.

It follows that for a weighted bipartite graph $G$, $\mathbb{E}(G)$ is a strictly monotonically increasing function of the numbers 
$b_{k}(G)~(k=0,1,\cdots,\lfloor \frac{n}{2}\rfloor)$. 
Thus, for instance, among all trees on $n$ vertices (weights of all edges equal to 1), the star $K_{1,n-1}$ uniquely attains the minimum energy and the path $P_n$ uniquely attains the maximum energy (see
\cite{I2}).

In general, 
we can  define a quasi-ordering relation ``$\preceq$'' for weighted bipartite graphs as  follows.

Let $G_1$ and $G_2$ be two weighted bipartite graphs of order $n$. If $b_{k}(G_1)\le b_{k}(G_2)$ for all $k$ with $1\le k\le \lfloor \frac{n}{2}\rfloor$, then we write $G_1\preceq G_2$. (Note that $b_{0}(G)=1$ for all weighted bipartite graphs $G$.) Furthermore, if $G_1\preceq G_2$ and there exists at least one index $j$ such that $b_{j}(G_1)< b_{j}(G_2)$, then we write $G_1\prec G_2$. If $b_{k}(G_1)= b_{k}(G_2)$ for all $k$, we write $G_1\approx G_2$. That there are nonisomorphic weighted bipartite graphs $G_1$ and $G_2$ with $G_1\approx G_2$ implies that ``$\preceq$'' is not a partial order, in general,  but a quasi-partial order.

From the Coulson integral formula (\ref{eq:coulson}), we obtain the important fact: 

\smallskip\noindent
{\it  If $G_1$ and $G_2$ are two weighted bipartite graphs  of order $n$, then $G_1\preceq G_2$ implies that $\mathbb{E}(G_1)\le \mathbb{E}(G_2)$ and $G_1\prec G_2$ implies that $\mathbb{E}(G_1)< \mathbb{E}(G_2)$}. 

We make use of this fact throughout this paper.
For more background on graph energy we refer the reader to \cite{I2}.

\section{The Minimum Energy Problems for multitrees and multiforests}

In this section, we give the solutions of the minimum energy problems for both the classes  ${\mathcal T}(n,m)$ and  ${\mathcal F}(n,m)$.

\begin{lemma}\label{lem:one}
Let $G$ be a weighted  tree. Then we have 
$$b_{1}(G)=\sum_{e\in E(G)}w(e)^2$$
where $w(e)$ denotes the weight of edge $e$.
\end{lemma}

\begin{proof} Let $\lambda_1,\cdots,\lambda_n$ be  the eigenvalues of $G$, and let $A(G)$ be the (weighted) adjacency matrix of $G$. Then from (\ref{eq:charpoly}) we have
\begin{eqnarray*}
b_{1}(G)&=& - \sum_{1\le i<j\le n}\lambda_i \lambda_j =-\frac {1} {2} \left (\left(\sum_{j=1}^n\lambda_j\right)^{2}-\sum_{j=1}^n\lambda_j^2 
\right )=\frac {1} {2} \sum_{j=1}^n\lambda_j^2\\ &=&\frac {1} {2} {\rm tr}(A(G)^2)=\frac {1} {2} \sum_{i=1}^n \sum_{j=1}^n a_{ij}^2=\sum_{e\in E(G)}w(e)^2.
\end{eqnarray*}
\end{proof}

Let $k\le m$ be a positive integer, and $r= \lfloor m/k \rfloor$. A set of $k$ positive integers $a_{1},a_{2},\cdots, a_{k}$ with sum $m$ is  {\it as equal as possible}  provided that  $a_{i}$ equals either $r$ or $r+1$ for each   $i=1,2, \cdots,k$.  A {\it star} $K_{1,p}$ is a tree of order $p+1$ with one vertex joined to all other vertices.

\vskip 0.5cm 

Using Lemma \ref{lem:one} we can obtain the following result.

\begin{theorem}\label{th:one}
 Let $S$ be the weighted star in $\mathcal{T}(n,m)$ whose weights are as equal as possible, and $T$ be any weighted tree in $\mathcal{T}(n,m)$ with $T\not= S$ $($as weighted graphs$)$. Then we have $S\prec T$ and hence $\mathbb{E}(S)<\mathbb{E}(T)$.
\end{theorem}

\begin{proof} First, we always have $b_{0}(S)=b_{0}(T)=1$. Also it is easy to see that the rank of the weighted adjacency matrix $A(S)$ of $S$ is 2. So $(n-2)$ of the eigenvalues of $S$ equal zero, and thus $b_{k}(S)=0 \le b_{k}(T)$ for all $k\ge 2$.  We show that either $b_1(S)<b_1(T)$ or, when $b_1(S)=b_1(T)$,  that $b_2(T)>0$. This will then prove the theorem.

Since  $T\not= S$ (as weighted graphs), $T$ is either not a star or is not as equally as possible weighted. Let the weight sequences of $S$ and $T$ be, respectively,
\[(s_1,s_2,\ldots,s_{n-1})\mbox{ and } (t_1,t_2,\ldots,t_{n-1}).\]
We consider the following two cases.

\textbf{Case 1:} $T$ is not weighted as equal as possible.

In this case we show that $b_1(S)<b_1(T)$, that is, by Lemma \ref{lem:one},
\begin{equation}\label{eq:case1}
\sum_{i=1}^{n-1}s_i^2<\sum_{i=1}^{n-1}t_i^2.
\end{equation}
Let $J_{n-1}$ equal  the matrix of order $n-1$ with all entries one. For arbitrary real numbers $x_1,x_2,\ldots,x_{n-1}$ summing to $m$, we  have: 
\begin{eqnarray*}
\sum_{1\le i<j\le n-1}(x_i-x_j)^2 &=&(x_1,x_2,\ldots,x_{n-1})((n-1)I_{n-1}-J_{n-1})(x_1,x_2,\ldots,x_{n-1})^t\\
&=& (n-1)(x_1^2+x_2^2+\cdots+x_{n-1}^2)-(x_1+x_2+\cdots+x_{n-1})^2\\ &=&(n-1)(x_1^2+x_2^2\cdots+x_{n-1}^2)-m^2.
\end{eqnarray*}
Since
\[\sum_{1\le i<j\le n-1}(s_i-s_j)^2<\sum_{1\le i<j\le n-1}(t_i-t_j)^2,\]
this implies that
\[b_{1}(S)=\sum_{i=1}^{n-1}s_i^2<\sum_{i=1}^{n-1}t_i^2=b_{1}(T)\]
that is, (\ref{eq:case1}) holds, and so  we have $S\prec T$.

\textbf{Case 2:} $T$ is weighted as equal as possible. Since $T\ne S$ as weighted trees, $T$ is not a star. 

In this case we have 
\[b_{1}(S)=\sum_{i=1}^{n-1}s_i^2=\sum_{i=1}^{n-1}t_i^2=b_{1}(T).\]
Since $T$ is not a star, it contains at least one $2$-matching. Hence
 $b_{2}(T)>0$.
% since  $T$ is not a star impling that $T$ contains some 2-matchings. Thus we also 
Thus we also have $S\prec T$ in this case. \end{proof}

The following theorem determines the minimum energy of the class  ${\mathcal T}(n,m)$ and the unique weighted tree in this class attaining this minimum energy.

\begin{theorem}\label{th:oneb}  Let $m= r(n-1)+t$, where  $r= \lfloor m/(n-1) \rfloor$ and $0\le t\le n-2$. Then we have: 

\noindent {\rm(a)} The multitree $S=K_{1,n-1}\in {\mathcal T}(n,m)$ whose weights are as equal as possible is the unique multitree in ${\mathcal T}(n,m)$ with minimum energy.

\noindent {\rm(b)} The value of the minimum energy in ${\mathcal T}(n,m)$ is $\mathbb {E}(S)=2\sqrt {mr+tr+t}$. 
\end{theorem}

\begin{proof}
The assertion (a) follows directly from Theorem \ref{th:one}.
From  Lemma \ref{lem:one}  we see  that 
$$b_{1}(S)=\sum_{e\in E(S)}w(e)^2= (n-1-t)r^{2}+ t(r+1)^{2}=(n-1)r^{2}+2tr+t=mr+tr+t.$$
On the other hand, the  characteristic polynomial of $S$ is:
$$\phi(S,x) = x^{n}- b_{1}(S)x^{n-2}.$$
So we have  $\mathbb {E}(S)=2\sqrt {b_{1}(S)}=2\sqrt {mr+tr+t}$.
\end{proof}

 If $G$ is a (weighted) graph with connected components $G_1,G_2,\ldots,G_l$, then we
write $G=(G_1,G_2,\ldots,G_l)$ with the order of the graphs being arbitrary.

The following theorem determines the minimum energy of the class  ${\mathcal F}(n,m)$ and the unique weighted forest in this class attaining this minimum energy.

\begin{theorem}\label{th:forest} Let $n,m$ be positive integers. Then we have:

\noindent {\rm (a)} If $m\le n-2$, then the weighted forest $F=(K_{1,m},K_{1},\cdots, K_{1})\in {\mathcal F}(n,m)$ whose edges all have weight one  is the unique minimum energy graph in ${\mathcal F}(n,m)$. In this case, the value of the minimum energy is  $\mathbb {E}(F)=2\sqrt {m}$. 

\noindent {\rm (b)}  If $m\ge n-1$, then the unique minimum energy weighted tree $S$ in ${\mathcal T}(n,m)$ given in Theorem $ \ref{th:oneb}$  is also the unique minimum energy weighted forest in ${\mathcal F}(n,m)$.  

\end{theorem}

\begin{proof} First we  prove assertion (a).
 We have $b_{1}(F)=m$ and  $b_{i}(F)=0$ for $i\ge 2$. Let $H\in  {\mathcal F}(n,m)$ with $H\not =F$. Then 
$$b_{1}(H)= \sum_{e\in E(H)}w(e)^{2}\ge  \sum_{e\in E(H)}w(e) = m = b_{1}(F)$$ with equality if and only if the weights of all edges of $H$ are one. So if the weight of some edge of $H$ is not one, then we have $b_{1}(H)>b_{1}(F)$, and thus $F\prec H$  and $\mathbb{E}(F)<\mathbb{E}(H)$.

On the other hand, if the weight of every edge of $H$ is one, then $H$ is a forest of order $n$ with $m$ edges. If $H\not =F$, then these $m$ edges of $H$ do not form a star, and so  
$b_{2}(H)>0=b_{2}(F)$. Thus we also have $F\prec H$ and $\mathbb{E}(F)<\mathbb{E}(H)$. 

For (b), let $H\in  {\mathcal F}(n,m)$ with $H\not =S$. If the weight sequence of $H$ is not the same as that of $S$, then similarly as in the proof of  Theorem \ref{th:one}, we can show that $b_{1}(S)<b_{1}(H)$. Thus  $S\prec H$ and $\mathbb{E}(S)<\mathbb{E}(H)$.

 If the weight sequence of $H$ is the same as that of $S$, then $H$ is a tree, and the result now follows from Theorem \ref{th:one}. 
\end{proof}

The following theorem shows that, if we consider the subclass of  ${\mathcal T}(n,m)$ where the weight sequence is fixed, then the star is still the unique minimum energy graph.

\begin{theorem}\label{th:new}
Let $T\in{\mathcal T}(n,m)$ with $T$ not equal to the star $K_{1,n-1}$. Let $S=K_{1,n-1}\in {\mathcal T}(n,m)$ whose weight sequence is the same as for $T$. Then  $S\prec T$ and thus $\mathbb{E}(S)<\mathbb{E}(T)$.
\end{theorem}

\begin{proof}
We have $b_{0}(S)=b_{0}(T)=1$, and $b_{1}(S)=b_{1}(T)$ by Lemma \ref{lem:one}, since $S$ and $T$ have the same weight sequence. Also by the same reason as in the proof of Theorem \ref{th:one}, we have  $b_{k}(S)=0 \le b_{k}(T)$ for all $k\ge 2$, and  $b_{2}(T)>0$ since  $T$ is not a star. So again we have  $S\prec T$. 
\end{proof}

\section{Some Results and Examples for the Maximum Energy Problems}

We first consider the maximum energy $\overline{\mathbb{E}}_F(n,m)$  over the class  ${\mathcal F}(n,m)$ of weighted forests and characterize those weighted forests in this class whose energy attains this maximum value.

\begin{theorem}\label{th:for1}
A weighted forest  $F$ in ${\mathcal F}(n,m)$ has the maximum energy in  ${\mathcal F}(n,m)$ if and only if each connected component of $F$ is $K_{1}$ $($a single vertex$)$ or $K_{2}$ $($two vertices joined by an edge$)$, and in this case $\mathbb {E}(F)=2m$. 
\end{theorem}

\begin{proof} It is straightforward to check that if each component of $F$ is $K_{1}$ or $K_{2}$, then $\mathbb {E}(F)=2m$, independent of how the weights are distributed on  the edges of $F$ (even if the weights are  general positive real  numbers).

On the other hand, if some component of $F$ is not $K_{1}$ or $K_{2}$, then from \cite{SHGD} we know that $\mathbb {E}(F)< 2m$ for all positively weighted forests of order $n$ with total weight sum $m$. Thus $\mathbb {E}(F)< 2m$ also holds for this integral weighted forest  $F\in {\mathcal F}(n,m)$. This prove the desired result.
\end{proof}

\begin{lemma}\label{lem:for1} Let $a=(a_{1},a_{2},\cdots,a_{n-1})=(m-n+2,1,\cdots,1)$ and $b=(b_{1},b_{2},\cdots,b_{n-1})$ be non-increasing positive integral vectors of dimension $n-1$, with $\sum_{i=1}^{n-1}a_{i}=\sum_{i=1}^{n-1}b_{i}=m$ and $a\not =b$. Then we have $\sum_{i=1}^{n-1}a_{i}^{2}>\sum_{i=1}^{n-1}b_{i}^{2}.$
\end{lemma}

\begin{proof}
Since $a\not =b$, we have $m-n+2>b_{1}\ge b_{2}\ge 2$. Let $c=(c_{1},c_{2},\cdots,c_{n-1})$ with 
$c_{1}=b_{1}+1$, $c_{2}=b_{2}-1$ and $c_{i}=b_{i}$ for all $i\ge 2$. Then we have $\sum_{i=1}^{n-1}c_{i}=m$ and $\sum_{i=1}^{n-1}c_{i}^{2}>\sum_{i=1}^{n-1}b_{i}^{2}$. This implies that $\sum_{i=1}^{n-1}b_{i}^{2}$ does not reach the maximum among all the  positive integral vectors $b\ne a$  of dimension $n-1$ the  sum of whose coordinates equals  $m$. So the corresponding maximum can only be reached by the vector $a$, and thus  $\sum_{i=1}^{n-1}a_{i}^{2}>\sum_{i=1}^{n-1}b_{i}^{2}$.  
\end{proof}

Using Lemma \ref{lem:for1},  we can obtain the maximum energy graph in the subclass of  ${\mathcal T}(n,m)$ where the underlying (unweighted) graph is fixed to be the star.

\begin{theorem}\label{th:onea}   The weighted star $S^*=K_{1,n-1}\in {\mathcal T}(n,m)$ with the weight sequence  $a=(a_{1},a_{2},\cdots,a_{n-1})=(m-n+2,1,\cdots,1)$  is the unique weighted star in ${\mathcal T}(n,m)$ with maximum energy among all the weighted stars in  ${\mathcal T}(n,m)$, and its energy is $\mathbb {E}(S^*)=2\sqrt {(m-n+2)^{2}+n-2}$.
\end{theorem}

\begin{proof}
Let $T\in {\mathcal T}(n,m)$ be a weighted star  whose weight sequence  is given by  $b=(b_{1},b_{2},\cdots,b_{n-1})$ different from the weight sequence $a$ of $S^{*}$. From  Lemma \ref{lem:for1}, we have 
$$b_{1}(S^*)=\sum_{i=1}^{n-1}a_{i}^{2}>\sum_{i=1}^{n-1}b_{i}^{2}=b_{1}(T)$$  
On the other hand, we also have $b_{i}(S^*)=b_{i}(T)=0$ for $i\ge 2$. So we have $T\prec S^*$ and hence $\mathbb{E}(T)<\mathbb{E}(S^*)$.
Finally, we have  $$\mathbb {E}(S^*) =2\sqrt {b_{1}(S^*)}  = 2\sqrt {\sum_{i=1}^{n-1}a_{i}^{2}} = 2\sqrt {(m-n+2)^{2}+n-2}$$.
\end{proof}

The maximum energy over the class ${\mathcal T}(n,m)$ of weighted trees on $n$ vertices with total weight $m$ appears to be very difficult, even if we restrict ourselves to the weighted paths in ${\mathcal T}(n,m)$.

%The maximum energy problem for the subclass of  ${\mathcal T}(n,m)$ where the underlying (unweighted) graph is fixed to be the path $P_{n}$ seems to be not so easy. The following is an example considering the case $n=4$. 

\begin{example}\label{ex:one}
{\rm Let $T\in  {\mathcal T}(4,m)$ be the weighted path $P_{4}$ of order 4 where the weights of its three edges are $a,b,c$ with $a+b+c=m$ and the edge with  weight $b$ is the middle edge of $P_{4}$. Then 
%by the weighted Sachs Theorem  we have
a simple computation yields 
$$\phi(T,x) = x^{4}-(a^{2}+b^{2}+c^{2})x^{2}+a^{2}c^{2}.$$
Let $y_1$ and $y_2$ be the roots of the quadratic equation
\[y^2-(a^2+b^2+c^2)y+a^2c^2=0.\]
Then the eigenvalues of $T$ are
$\pm \sqrt{y_1}$ and $\pm \sqrt{y_2}$. Thus
\[\mathbb{E}(T)=2(\sqrt{y_1}+\sqrt{y_2})\] implying that
\[\frac{\mathbb{E}(T)^2}{4}=y_1+y_2+2\sqrt{y_1y_2}=a^2+b^2+c^2+2ac.\]
Hence
\begin{equation}\label{eq:formula}
\mathbb{E}(T)=2\sqrt{(a+c)^2+b^2}=2\sqrt{(m-b)^2+b^2}.\end{equation}
%From this we can compute that 
%$$\mathbb {E}(T)= 2\sqrt {(a+c)^{2}+b^{2}}= 2\sqrt {(m-b)^{2}+b^{2}}$$
(Thus in the case $n=4$, the energy does not depend on the individual
values of $a$ and $c$.)
In the interval $1\le b\le m-2$,
the function in (\ref{eq:formula})  reaches the maximum when $b=1$, and so the maximum energy equals
$2\sqrt{(m-1)^2+1}$ and is attained for all positive integral weight sequences $a,b,c$ where $a\ge 1$, $b=1$, $c=m-1-a$.}
\end{example}

In contrast, we offer the following conjecture.

\begin{conjecture} Let $n\ge 5$ and let $m\ge n$.
The path in ${\mathcal T}(n,m)$ with weight sequence $(m-n+2,1,\ldots,1)$ where  the weight of one of the pendent edges equals $m-n+2$ is the unique tree in ${\mathcal T}(n,m)$ with maximum energy.
\end{conjecture}

If we restrict ourselves to the weighted trees in ${\mathcal T}(n,m)$, then the conclusion of the conjecture holds, as we now show.
Let $P_n^*$ be the weighted path $P_n$ in ${\mathcal T}(n,m)$ whose edges have weights $a,1,\ldots,1$ where $a=m-n+2$ is the weight of a pendent edge of $P_n$.

\begin{theorem}\label{th:gong} Let $n\ge 3$,  let $m\ge n$, and let $a=m-n+2\ge 2$.
 Let ${\mathcal T}(n,m;a,1,\ldots,1)$ be the set of all  trees  in ${\mathcal T}(n,m)$
with  weight sequence $(a,1,\cdots,1)$.  Then for all $T\in {\mathcal T}(n,m;a,1,\ldots,1)$ with $T\ne P_n^*$,  we have $T\prec P_n^*$. 
\end{theorem}

\begin{proof}
 Let $T\in {\mathcal T}(n,m;a,1,\ldots,1)$ with $T\ne P_n^*$,.
We prove that $T\prec P_n^*$ by induction on $n$. 
If $n=3$, then the result is obviously true since there is only one graph in the set  ${\mathcal T}(3,m;a,1,\cdots,1)$.
If $n=4$, then the underlying unweighted graph of $T$ is either the star $K_{1,3}$ or the path $P_{4}$. In the former case, the result follows from Theorem \ref{th:new}. In the latter case, the result follows from Example \ref{ex:one}..

Now we assume that $n\ge 5$. Since $T$ contains at least two pendent edges, there exists a pendent edge $uv$, where $v$ is a  pendent vertex,  having weight 1. Then we have 

$$b_{k}(T)=b_{k}(T-v)+b_{k-1}(T-v-u) \quad (k\ge 1)$$
By the induction assumption, we have 
\begin{equation}\label{eq:new*}
T-v\preceq P^{*}_{n-1}.\end{equation}
Now we show that 
\begin{equation}\label{eq:new**}
T-v-u\preceq P^{*}_{n-2}.\end{equation}
If $T-v-u$ is connected, then (\ref{eq:new**}) holds by induction: if $T-v-u$ does not contain an edge of weight $a$, we can change the weight of one edge from 1 to $a$, and then use the induction assumption. If $T-v-u$ is not connected, we can add some edges to $T-v-u$ to reduce the proof to the connected case. Thus  (\ref{eq:new**}) holds in both cases.

Finally, we show that at least one of the two quasi-order relations (\ref{eq:new*}) and (\ref{eq:new**}) is strict.

Let $Q_{n}$ be the tree of order $n$ obtained from $P_{n-1}$ by adding a new pendent edge at a quasi-pendent vertex of $P_{n-1}$ (here a quasi-pendent vertex is vertex adjacent to some pendent vertex). Since $n\ge 5$,  $Q_{n}$ contains a unique pendent edge $\alpha $ which is not adjacent to any other pendent edge of $Q_{n}$. Let $Q_{n}^{*}$ be the weighted tree obtained from $Q_{n}$ by assigning weight $a$ to $\alpha $ and weights 1 to all other edges. 

If $T\not = Q_{n}^{*}$, then since $T\not = P_{n}^{*}$,  there exists at least one pendent edge $uv$ (with pendent vertex $v$) of $T$ with weight 1 such that $T-v\not =P^{*}_{n-1}$. If we take this pendent edge $uv$, then the quasi-order relation (\ref{eq:new*}) is strict by induction. 

If $T= Q_{n}^{*}$, then $T-v-u$ is not connected. So in this case the quasi-order relation (\ref{eq:new**}) is strict. 

This proves that at least one of the two quasi-order relations (\ref{eq:new*}) and (\ref{eq:new**}) is strict. Thus we have  $T\prec P^{*}_{n}$ completing the inductive proof of the theorem. 
 \end{proof}

\begin{corollary} Let $n\ge 4$,  let $m\ge n$, and let $a=m-n+2\ge 2$.
 The weighted path  $ P_n^*$
is the unique weighted tree in ${\mathcal T}(n,m;a,1,\ldots,1)$ with maximum  energy.
\end{corollary}

%if $n\ge 5$, we {\it conjecture}  that the maximal energy graph in this subclass is the weighted path $P_{n}$ where the weights of all the non-pendent edges of $P_{n}$ equal  $1$ and the weights of the two pendent edges of $P_{n}$ are as unequal as possible, that is, for weight sequences (in the order of the edges of the path) equal to
%$1,\ldots,1,m-(n-2)$. For example, when $n=5$ and the weight sequence is $1,1,1,1,6$, the energy equals $14.9784$, while for the weight sequence $3,1,1,1,4$, the energy equals $14.5604$.

\section{The Case of $(0,1)$ Weights}

In this section we assume that the weights are 0 and 1 with $m$ 1s and $(n-1-m)$ 0s. In this case, ${\mathcal T}(n,m)$ is the set 
${\cal F}_n^k$ of forests with $n$ vertices and $k=n-m\ge 1$ connected components (trees). We will determine the (unique) forest in ${\cal F}_n^k$ with minimum energy and the  (unique) forest with maximum energy.

The following lemma is due to Gutman \cite{I}.

\begin{lemma}\label{lem:star} 
The energy of the star $S_n=K_{1,n-1}$   equals $2\sqrt{n-1}$, and $S_n$ is the unique tree of   minimum energy among  all trees with $n$ vertices.
\end{lemma}

The following lemma contains an elementary inequality.

\begin{lemma}\label{lem:sqrt}
Let $a_1,a_2,\ldots,a_k$  be nonnegative real numbers. Then
\[\sqrt{a_1}+\sqrt{a_2}+\cdots +\sqrt{a_k}\ge \sqrt{a_1+a_2+\cdots+a_k},\]
with equality if and only if at most one of $a_1,a_2,\ldots,a_k$ is nonzero.
\end{lemma}

% If $G$ is a graph with connected components $G_1,G_2,\ldots,G_l$, then we
%write $G=(G_1,G_2,\ldots,G_l)$ with the order of the graphs being arbitrary.

\begin{theorem}\label{th:for1}
The forest $F_n^k= (K_{1,n-k}, K_1,\ldots, K_1)$ with $(k-1)$ isolated vertices is the unique forest in ${\mathcal F}_n^k$ with  minimum energy $2\sqrt{n-k}$.
\end{theorem}

\begin{proof}
Let $T_1,T_2,\ldots,T_k$ be the connected components of a forest $F\in {\mathcal F}_n^k$ where $T_i$ has order $n_i$ $(i=1,2,\ldots,k)$.
Then
\begin{eqnarray*}
E(F)=\sum_{i=1}^k E(T_i)&\ge& \sum_{i=1}^k E(S_{n_i}) \quad \mbox{ (by Lemma \ref{lem:star})}\\
&=& 2\sum_{i=1}^k \sqrt{n_i-1}\\
&\ge & 2\sqrt{\sum_{i=1}^k (n_i-1)}\quad \mbox{ (by Lemma \ref{lem:sqrt})}\\
&=& 2\sqrt{n-k}\\
&=& E(F_n^k).\end{eqnarray*}
By Lemmas \ref{lem:star} and  \ref{lem:sqrt}, $E(F)=E(F_n^k)$ if and only if $F=F_n^k$.
\end{proof}

We now consider the maximum energy of forests in ${\mathcal F}_n^k$. It turns out there are two cases to consider according to whether $2k$ is larger or smaller than $n$. In both cases we identify the forest with maximum energy.

Denoting again a  path with $n$ verticses by  $P_n$.
we have the  following lemma proved in \cite{IS}.

\begin{lemma}\label{lem:path} 
For $n\ge 3$,
\[P_k\cup P_{n-k}\prec P_2\cup P_{n-2} \quad (k\ne 2,n-2).\]
\end{lemma}

\begin{theorem}\label{th:max1}
If $2k\ge n$, the forest  $M_n^k=(P_2,\ldots,P_2,K_1,\ldots,K_1) \in {\cal F}_n^k$, where  $P_2$ occurs $(n-k)$ times and  $K_1$ occurs $(2k-n)$ times, is the unique forest in ${\mathcal F}_n^k$ with maximum energy, and its energy is $2(n-k)$. 
\end{theorem}

\begin{proof}
It is easy to see that  $\mathbb {E}(M_n^k)=2(n-k)$. 
On the other hand, if $F \in {\cal F}_n^k$ is not $M_n^k$, then from \cite{SHGD} we know that $\mathbb {E}(F)< 2(n-k)$ since $F$  contains $(n-k)$ edges (we can view that each edge has weight one). This prove the desired result.
\end{proof}

\begin{theorem}\label{th:max2}
If $2k< n$, the forest  $P_n^k=(P_2,\ldots,P_2,P_{n-2k+2}) \in {\cal F}_n^k$, where  $P_2$  occurs $(k-1)$ times, is the unique forest in ${\mathcal F}_n^k$ with maximum energy.
\end{theorem}

\begin{proof}
Let $T_1,T_2,\ldots,T_k$ be the connected components of a forest $F\in {\mathcal F}_n^k$ where $T_i$ has order $n_i$ $(i=1,2,\ldots,k)$.
Suppose e.g. neither $T_1$ nor $T_2$ equals $P_2$.
Since   a path is the unique graph with  maximum energy among all trees with the same number of vertices,  and using Lemma \ref{lem:path} we get
\[(T_1,T_2)\preceq (P_{n_1},P_{n_2})\preceq (P_2,P_{n_1+n_2-2})\]
with at least one of $\preceq$ equal to $\prec$. Thus, we may replace $T_1$ and $T_2$ in $F$ with $P_2$ and $P_{n_1+n_2-2}$ and obtain a forest in ${\mathcal F}_n^k$ with larger energy.  Hence $F$ is not a forest of maximum energy in ${\mathcal F}_n^k$. It follows that a forest in ${\mathcal F}_n^k$ with maximum energy has $(k-1)$ 
components equal to $P_2$. 
%Since a path is the unique graph with  maximum energy among all trees with the same number of vertices, the lemma now follows.
\end{proof}

\bibliographystyle{plain}
%\bibliography{reference}

\begin{thebibliography}{1}

\bibitem{CA}
C.A.
\newblock {Coulson}.
\newblock On the calculation of the energy in unsaturated
hydrocarbon molecules.
\newblock {\em Proc. Cambridge Phil. Soc.,} \textbf{36} (1940), 201--203.

\bibitem{DMH}
D. Cvetkovi\'c, M. Doob and H. Sachs.
\newblock {Spectra of graphs}.
\newblock {\em Academic Press,} New York 1980.

\bibitem{CRS}
D.~Cvetkovi\'c, P.~Rowlinson, S. Simi\'c.
\newblock {An Introduction to the Theory of Graph Spectra}.
\newblock {London Mathematical Society Student Texts 75.}
\newblock {\em Cambridge University Press,} Cambridge 2010.

\bibitem{I}
I. Gutman.
\newblock {Acyclic systems with extremal H$\ddot{u}$ckel $\pi$-electron energy}.
\newblock {\em Theoret. Chim. Acta (Berlin),} \textbf{45} (1977), 79--87.

\bibitem{I2}
I. Gutman.
\newblock {The energy of a graph: Old and new results}.
\newblock in: A. Betten, A. Kohnert, R. Laue, A. Wasserman (Eds), {\em Algebraic Combinatorics and Applications, Springer-Verlag, Berlin,} 2001, pp. 196--211.

%\bibitem{IM}
%I. Gutman, M. Mateljevi$\acute{c}$.
%\newblock {Note on the Coulson integral formula}.
%\newblock {\em J. Math. Chem.,} \textbf{39} (2006), 259--266.

\bibitem{IO}
I. Gutman, O.E. Polansky.
\newblock {Mathematical Concepts in Organic Chemistry}.
\newblock {\em Springer Berlin,} 1986.

\bibitem{IS}
I. Gutman, J.Y.~Shao.
\newblock {The energy change of weighted graphs.}
\newblock {\em Linear Algebra Applic.} \textbf{435} (2011), 2425--2431.



%\bibitem{MI}
%M. Mateljevi$\acute{c}$, I. Gutman.
%\newblock {Note on the Coulson and Coulson-Jacobs integral formulas}.
%\newblock {\em MATCH Commun. Math. Comput. Chem.,} \textbf{59} (2008), 257--268.

%\bibitem{MVI}
%M. Mateljevi$\acute{c}$, V. Bo$\check{z}$in, I. Gutman.
%\newblock {Energy of a polynomial and the Coulson integral formula}.
%\newblock {\em J. Math. Chem.,} \textbf{48} (2010), 1062--1068.

%\bibitem{IJL}
%I. Pe$\tilde n$a, J. Rada. 
%\newblock {Energy of digraphs}.
%\newblock {\em Lin. Multilin. Algebra,} \textbf{56} (2008), 565--579.

\bibitem{SHGD}
J.Y. Shao, F. Gong and Z.B. Du.
\newblock {The extremal energies of weighted trees and forests with fixed total weight sum}.
\newblock {\em MATCH Commun. Math. Comput. Chem.,} to appear.




\end{thebibliography}

\end{document}